\documentclass[sn-mathphys,Numbered]{sn-jnl}

\usepackage{graphicx}%
\usepackage{multirow}%
\usepackage{amsmath,amssymb,amsfonts}%
\usepackage{amsthm}%
\usepackage{mathrsfs}%
\usepackage[title]{appendix}%
\usepackage{xcolor}%
\usepackage{textcomp}%
\usepackage{manyfoot}%
\usepackage{booktabs}%
\usepackage{algorithm}%
\usepackage{algorithmicx}%
\usepackage{algpseudocode}%
\usepackage{listings}%

\theoremstyle{thmstyleone}%
\newtheorem{theorem}{Theorem}
\theoremstyle{thmstyletwo}%

\theoremstyle{thmstylethree}%

\newtheorem{lemma}{Lemma}

\begin{document}

\title[Article Title]{The Non-zonal Rossby-Haurwitz Solutions of the 2D Euler Equations on a Rotating Ellipsoid}


\author*{\fnm{Chenghao} \sur{Xu}}\email{cx474@nyu.edu}

\affil{\orgdiv{Courant Institute of Mathematical Sciences}, \orgname{New York University}, \orgaddress{\street{251 Mercer Street}, \city{New York}, \postcode{10012}, \state{New York}, \country{USA}}}


\abstract{
    In this article, we investigate the incompressible 2D Euler equations on a rotating biaxial ellipsoid, which model the dynamics of the atmosphere of a Jovian planet. 
    We study the non-zonal Rossby-Haurwitz solutions of the Euler equations on an ellipsoid, while previous works only considered the case of a sphere.  
    Our main results include: 
    the existence and uniqueness of the stationary Rossby-Haurwitz solutions; 
    the construction of the traveling-wave solutions; 
    and the demonstration of the Lyapunov instability of both the stationary and the traveling-wave solutions.
    }




\maketitle

\bmhead{Acknowledgements}
The author is grateful to Professor Pierre Germain for introducing the topic and providing illuminating suggestions.
The author thanks Professor Vlad Vicol for his help with article organization.
The author is also grateful to Professor Katherine Zhiyuan Zhang for her consulting assistance.

\section{Introduction}

The incompressible Euler equations on a rotating 2D manifold $M$ embedded in $\mathbb{R}^3$ has been widely-studied. In this setting, the Euler equations are taken as a model of the behavior of the atmosphere of a rotating planet.\footnote{\label{note1}It is an interesting topic for further investigation to derive the 2D Euler equations from a physics perspective on a rotating ellipsoid that models a Jovian planet, or on a more general manifold. 
The case that has been done in the literature is when $M$ is a sphere (see Constantin and Germain~\cite{pierre2022}, Constantin and Johnson~\cite{constantin2022propagation}, Gill~\cite{gill1982atmosphere}).
Even though, scholars have taken the 2D Euler equations as the governing equations of planetary atmospheric flows in the case when $M$ is an ellipsoid (see Tauchi and Yoneda~\cite{tauchi2022existence}) or a rotationally symmetric manifold (see Talyor~\cite{taylor2016euler}).}
Of great interest is the analysis of the solutions of the Euler equations, together with their stability properties. 

Taking $M$ as a perfect sphere seems to be the most natural approach (see Cheng and Mahalov~\cite{cheng2013euler}, Constantin and Germain~\cite{pierre2022}). 
In this case, two classes of solutions are well-studied in the literature, which are zonal solutions and non-zonal Rossby-Haurwitz solutions. 

Zonal solutions represent the arrangement of the atmospheric band structure of outer planets, which consists of alternating westward and eastward winds. 
The complex atmospheric dynamics of a Jovian planet can be viewed as a background zonal solution presenting fluctuations: some stable, while other are unstable and may develop wildly over time. 
Studying the stability properties of these zonal solutions may offer a physical insight into the atmosphere science.
Zonal solutions are automatically stationary and their stability properties are developed and stated for instance in Constantin and Germain~\cite{pierre2022}, or Marchiorio and Pulvirenti~\cite{marchioro2012mathematical}.
Moreover, due to the rich symmetry of the sphere, some non-zonal solutions can be obtained from zonal solutions by utilizing the invariance of the stationary Euler equations through the action of $\mathbb{O}(3)$.

On a unit sphere, classical non-zonal Rossby-Haurwitz solutions are complete nonlinear and non-trivial solutions, which were first found by Craig~\cite{craig1945solution}, of the Euler equations, obtained by Rossby~\cite{rossby1939relation} and Haurwitz~\cite{haurwitz1940motion}. 
Non-zonal Rossby-Haurwitz solutions can be either stationary or traveling, with the latter being derived from the former.
They are also the only known non-trivial solutions of the Euler equations with explicit expressions.
It is widely recognized that the corresponding Rossby-Haurwitz waves contribute significantly to the atmosphere dynamics. 
For instance, some non-zonal Rossby-Haurwitz waves of degree $2$ can be predominant in the atmosphere on a Jovian planet (see Dowling~\cite{dowling1995dynamics}). 
Their stability properties are crucial. 
For example, one of the main reasons for the difficulty in making accurate long-term weather forecasts is the instability of these waves (see Bénard~\cite{benard2020stability}).
In particular, Constantin and Germain~\cite{pierre2022} studied the non-zonal Rossby-Haurwitz solutions and proved their Lyapunov instability. 
Furthermore, owning to the abundant symmetries of the sphere, Cao, Wang, and Zuo~\cite{cao2023stability} proved that all the Rossby-Haurwitz solutions of degree 2 are orbitally stable.

However, modeling a real-world planet, even one with small eccentricity like the Earth, as a perfect sphere is inaccurate. An ellipsoidal model can provide a much more accurate representation of the oblate-spherical geometry.
Along this research direction, Constantin and Johnson~\cite{constantin2021propagation}~\cite{constantin2021modelling} derived the leading-order 3D compressible Navier–Stokes equations of the atmospheric flows on a rotating ellipsoid, via a thin-shell approximation based on the Earth's atmospheric and geographical data. 
These works may inspire future research on deriving the 2D Euler equations on a rotating ellipsoid for modeling an outer planet (see relation to footnote~\footref{note1}).

Furthermore, a spherical model may diverge significantly from the actual shape of outer planets such as Jupiter or Saturn.
A fast rotating Jovian planet usually has a relatively large flattening rate.
I.e., Jupiter deviates much from a perfect sphere by flattening at the poles and bulging at the equator (see Berardo and Wit~\cite{berardo2022effects}). 
Saturn has a large flattening rate about 0.1 (see Elkins-Tanton~\cite{elkins2006jupiter}), and Haziot~\cite{haziot2022spherical} revealed that a spherical model turned to be unsuitable for flows on Saturn.
Therefore, it is necessary to use a biaxial ellipsoid model that provides a better approximation of the shape of an outer planet. 
Sometimes, an ellipsoidal model can make a crucial difference. 
For instance, Tauchi and Yoneda~\cite{tauchi2022existence} studied the mechanisms behind stable multiple zonal jet flows, such as the famous Great Red Spot on Jupiter, by analyzing the 2D Euler equations on an ellipsoid from a differential geometry perspective.
However, their arguments do not hold when applied to a sphere. 
Additionally, Taylor~\cite{taylor2016euler} studied the 2D Euler equations on a general rotationally symmetric manifold to gain a better understanding of the planetary atmosphere, with a particular focus on the stability analysis of the zonal solutions.
Inspired by these works, this article investigates the case when $M$ is an ellipsoid and considers the 2D Euler equations~\eqref{Ew} as the governing equations.
As expected, the equation~\eqref{Ew} coincides with the 2D Euler equations in Taylor~\cite{taylor2016euler}.

\begin{figure}[h]%
\centering
\includegraphics[width=0.9\textwidth]{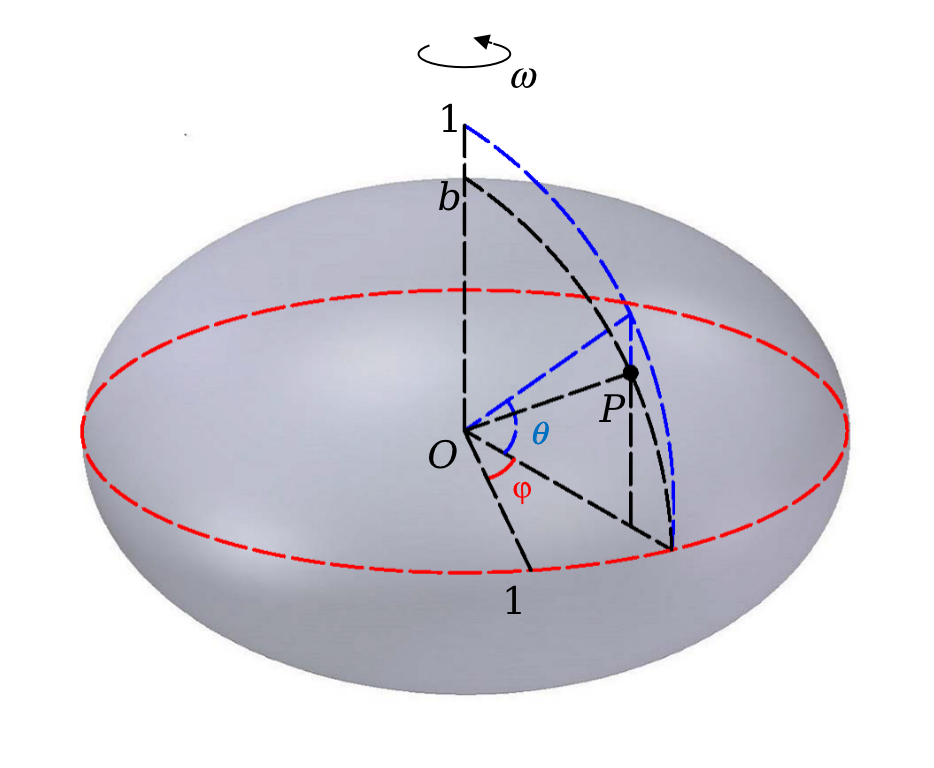}
\caption{A biaxial ellipsoid rotating with angular velocity $\omega$, whose major axis equals 1 and minor axis equals $b<1$. A point $P$ on its surface can be parameterized by $(\varphi,\theta)$. The equator is on $\theta=0$; the North pole is at $\theta=\frac{\pi}{2}$, while the South pole is at $\theta=-\frac{\pi}{2}$.}\label{Figure1}
\end{figure}

On the surface of a rotating biaxial ellipsoid (see Figure~\ref{Figure1}), with the major axis equal to $1$ and the minor axis equal to $b < 1$, the incompressible Euler equations can be expressed in terms of the stream function $\psi$ as
\begin{equation}\label{Ew}
     \left( \partial_t + \frac{1}{\cos\theta\sqrt{\sin^2\theta+b^2\cos^2\theta}}[-\partial_\theta\psi\partial_\varphi + \partial_\varphi\psi\partial_\theta] \right) \left(\Delta\psi+\frac{2\omega\sin\theta}{\sqrt{\sin^2\theta+b^2\cos^2\theta}}\right) 
     =0. \nonumber \tag{$\mathcal{E}_\omega$}
\end{equation}
Here, the polar coordinate $(\varphi, \theta) \in \left[-\pi,\pi\right)\times\left[-\frac{\pi}{2}, \frac{\pi}{2}\right]$ is used to parameterize the surface $\mathbb{S}^2$ as $(x,y,z) = (\cos\varphi\cos\theta, \sin\varphi\cos\theta, b\sin\theta)$. 
In equation~\eqref{Ew}, $\omega$ is the rotating speed of the ellipsoid and $\Delta$ is the Laplace-Beltrami operator on $\mathbb{S}^2$. 
Consequently, the stationary Euler equations become
\begin{equation}\label{Ew_stationary1}
     [-\partial_\theta\psi\partial_\varphi + \partial_\varphi\psi\partial_\theta]\left(\Delta\psi+\frac{2\omega\sin\theta}{\sqrt{\sin^2\theta+b^2\cos^2\theta}}\right)=0.
\end{equation}
Notice that both local and global well-posedness on $H^s$ $(s\geq 2)$ can be guaranteed (see Taylor~\cite{taylor2016euler}). 

On an ellipsoid, the stream function $\psi$ of a zonal solution only depends on the latitude angle $\theta$, namely, $\psi = \psi (\theta)$.
Notice that any zonal solution is stationary and solves~\eqref{Ew_stationary1}.
Taking $M$ to be a rotationally symmetric surface, Talyor~\cite{taylor2016euler} proved the stability results for zonal solutions, including both linear stability criteria (Rayleigh’s and Fjortoft’s) and nonlinear stability criterion (Arnold's).
As a special case, these properties can be inherited by an ellipsoid.

Nevertheless, non-zonal Rossby-Haurwitz solutions of the 2D Euler equations on a rotating ellipsoid were not much studied in the literature.
Thus, we establish theories and analyze the stability properties of these solutions in this article.

In the ellipsoidal setting, we analogously propose stationary non-zonal Rossby-Haurwitz solutions to be
\begin{equation}\label{RW_solu_intro}
    \psi = g(\theta) + Y_l^m(\varphi, \theta), 
\end{equation}
which solves~\eqref{Ew_stationary1}.
Here, $g(\theta)$ is some specific function in $C^{3,\alpha}\left(\left(-\frac{\pi}{2},\frac{\pi}{2}\right)\right)$ (see Section~\ref{Existence}), and $Y_l^m(\varphi, \theta)$ belongs to the $(l,m)$-th eigenspace of $-\Delta$ associated with eigenvalue $\lambda_{l,m}$, where $l \in \mathbb{N}$ and $m \in \{-l,\dots,l\}\backslash\{0\}$ (see Section~\ref{Laplace_elli}). 
We will show the existence, uniqueness and Lyapunov instability of the proposed solution in Section~\ref{Sec_Rossby}. 
Notice that when $b=1$, these solutions reduce to classical Rossby-Haurwitz solutions on a sphere. In this case, $g(\theta)$ becomes $\alpha\sin\theta$ for some constant $\alpha$, and $Y^m_l$ reduces to a linear combination of spherical harmonics of degree $l$ (see Section~\ref{Rossby_sphere}).  

Travelling-wave Rossby-Haurwitz solutions can be constructed through non-zonal stationary Rossby-Haurwitz solutions. Specifically, the solutions traveling with speed $c$ are constructed to be
\[
    \psi_c(\varphi, \theta ,t) = g(\theta) + c \lambda_{l,m} f(\theta) + Y_l^m(\varphi-ct, \theta),
\]
where $f(\theta)$ is some specific function in $C^{3,\alpha}\left(\left(-\frac{\pi}{2},\frac{\pi}{2}\right)\right)$, which is formally defined in~\eqref{Integral_f}.
As expected, these time dependent solutions solve the Euler equations~\eqref{Ew}. For a given $c$ and $Y_l^m$, we will demonstrate the existence and uniqueness of the traveling-wave solution $\psi_c$ and its Lyapunov instability in Section~\ref{Sec_Rossby}.

\subsection{Main results}

The stream functions of stationary Rossby-Haurwitz solutions on a sphere are
\[
\psi = \alpha\sin\theta + Y_l(\varphi, \theta), \quad \alpha = \frac{2\omega}{2-l(l+1)},
\]
where $Y_l$ belongs to the $l^{th}$ eigenspace of $\Delta$ on a unit sphere ($l\geq2$). However, they fail to be stationary solutions of~\eqref{Ew} on an ellipsoid. 
Naturally, we propose a modification 
\begin{equation}\label{main_res_propose}
    \psi = g(\theta) + Y_l^m(\varphi, \theta), \quad \underset{\theta\to\pm\frac{\pi}{2}}{\lim} g'(\theta)= 0,
\end{equation}
where $g(\theta) \in C^{3,\alpha}\left(\left(-\frac{\pi}{2},\frac{\pi}{2}\right)\right)$ and $Y_l^m$ belongs to the $(l,m)$-th eigenspace of $\Delta$ on $\mathbb{S}^2$ (see Section~\ref{propose_elli_Rossby}) to be the solutions of the stationary Euler equations~\eqref{Ew_stationary1}. 

In Theorem~\ref{thm_existence}, we prove the existence of $g$ in~\eqref{RW_solu_intro} that solves the ODE~\eqref{ODE1} with the boundary conditions~\eqref{boud_cond}. As a result, $\psi = g(\theta) + Y_l^m(\varphi,\theta)$ solves the stationary Euler equations~\eqref{Ew_stationary1}
\[
[-\partial_\theta\psi\partial_\varphi + \partial_\varphi\psi\partial_\theta]\left(\Delta\psi+\frac{2\omega\sin\theta}{\sqrt{\sin^2\theta+b^2\cos^2\theta}}\right)=0,
\]
with the boundary conditions~\eqref{Boundary_condition_psi}. 

Unlike on a unit sphere, there is no exact formula for Rossby-Haurwitz solutions on an ellipsoid. 
Instead, the proof of the existence requires solving a second order ODE~\eqref{ODE1} on $\left(-\frac{\pi}{2}, \frac{\pi}{2}\right)$ with regular singularities and Neumann conditions on both boundaries. 
The approach is to rewrite the ODE as a Volterra integral equation (VIE) and utilize the Variation-of-Constants formula to show the existence of a smooth solution on $\left(-\frac{\pi}{2}, 0\right]$ that vanishes at the origin and satisfies the left Neumann condition. 
Then through an odd extension, a solution can be constructed on the whole interval. 
Lastly, the smoothness of the solution at the origin can be guaranteed by exploiting a special property of the ODE, i.e. if a solution vanishes at the origin, its second derivative also vanishes at the origin.

In Theorem~\ref{thm_uniquess}, we prove the uniqueness (up to a constant) of $\psi = g(\theta) + Y_l^m(\varphi,\theta)$ for a given $Y_l^m$. As a result, the corresponding velocity field $U=J \operatorname{grad}\psi$ is unique. 

The uniqueness can be proved by contradiction. 
Specifically, under the non-uniqueness assumption, the difference (up to a constant) of two solutions $\Tilde{g}(\theta)$ must fall into the $(l,m)$-th eigenspace of $\Delta$. However, this is not true by the spectral theory of $\Delta$ on $\mathbb{S}^2$. 

In Theorem~\ref{thm_travelling}, we construct the travelling-wave Rossby-Haurwitz solutions
\begin{equation}\label{Travelling_main_results}
    \psi_c(\varphi, \theta ,t) = g(\theta) + c \lambda_{l,m} f(\theta) + Y_l^m(\varphi-ct, \theta),
\end{equation}
that solve the Euler equations~\eqref{Ew}, where $c\in\mathbb{R}$ is the speed and $f$ is some function defined in~\eqref{Integral_f}.

We seek for non-stationary solutions of~\eqref{Ew} travelling with speed $c$. Plugging $\psi_c(\varphi,\theta,t)$ of the form~\eqref{Travelling_main_results} into the Euler equations~\eqref{Ew}, we end up with a complicated second order ODE~\eqref{Integral_f} that is similar to the ODE~\eqref{ODE1} in Theorem~\ref{thm_existence}. Thus, the ODE~\eqref{Integral_f} can be solved by a similar approach and the solution is $f$.

In Theorem~\ref{thm_tra_uniq}, we state the uniqueness of the travelling-wave solution for a given speed $c$ and function $Y_l^m$.
The uniqueness can be proved in a similar way as in Theorem~\ref{thm_uniquess}, which is straightforward and thus omitted in this article. 

In Theorem~\ref{Ins_Ross}, we prove the Lyapunov instability for both travelling-wave and stationary ($c=0$) non-zonal Rossby-Haurwitz solutions. 
Specifically, for a given solution $\psi_c$, a sequence of travelling-wave solutions ${\psi^n_c}$ with initial data ${\psi^n_c}(0) \to \psi_c(0)$ are constructed, such that
\[
\underset{n\to\infty}{\liminf} \left\{ \sup_{t>0}\left\|{\psi^n_c}(t)-\psi_c(t)\right\|^2_{L^2(\mathbb{S}^2,d\sigma)} \right\} \geq \epsilon > 0.
\]
for some positive $\epsilon$.

The constructed solutions ${\psi^n_c}(t)$ travel with speed $c+\frac{1}{n}$, which exceed $\psi_c(t)$ a little bit. 
By formula~\eqref{Travelling_main_results}, one can show the initial data $\psi^n_c(0)$ converge to $\psi_c(0)$. 
However, although the traveling speeds also converge to $c$, the tiny differences will be amplified by time $t$, which leads to the Lyapunov instability.
This can be shown by expanding both ${\psi_c^n}(t)$ and $\psi_c(t)$ in terms of the basis of the eigenspace of $\Delta$ and choosing a proper time $t$.

\subsection{Organization of this article}

In Section~\ref{Sec_2}, we start with deriving the Euler equations~\eqref{Ew} on a rotating ellipsoid. 
In Section~\ref{Laplace}, we present the spectral theory of $\Delta$ on a sphere and on a biaxial ellipsoid. 
After a briefly introduction of classical Rossby-Haurwitz solutions on a sphere (see Section~\ref{Rossby_sphere}), we innovatively propose stationary Rossby-Haurwitz solutions on a rotating ellipsoid (see Section~\ref{propose_elli_Rossby}). 
Then, we demonstrate their existence (in Section~\ref{Existence}) and uniqueness (in Section~\ref{Uniqueness}). 
In Section~\ref{sec_tra}, we construct the travelling-wave Rossby-Haurwitz solutions and state the uniqueness. 
Finally, in Section~\ref{sec_instability}, we show the Lyapunov instability of the non-zonal Rossby-Haurwitz solutions.

\section{Derivation of the incompressible Euler equations on a rotating ellipsoid}\label{Sec_2}

For a biaxial ellipsoid with major axis $1$ and minor axis $b$ $(b<1)$, the standard coordinate chart we use here is 
\begin{equation}\label{coor_chart}
    (\varphi, \theta) \in (-\pi,\pi)\times\left(-\frac{\pi}{2},\frac{\pi}{2}\right) \mapsto (\cos\varphi\cos\theta, \sin\varphi\cos\theta, b\sin\theta) \in \mathbb{S}^2.
\end{equation}
The singularity introduced by the coordinate chart at the poles can be resolved by taking smoothness into account (see relation~\eqref{boundary_condition_g}). The double-valued ambiguity at $\varphi=\pm\pi$ can be handled by assuming a periodic dependence on variable $\varphi$ (see relation~\eqref{boundary_condition_phi}). 

For $p \in \mathbb{S}^2\backslash\{N,S\}$, the tangent space $T_p\mathbb{S}^2$ has a basis $\left\{\mathbf{e_\varphi}, \mathbf{e_\theta}\right\}$ which is
\[
\left\{\frac{1}{\cos\theta}\partial_\varphi, \frac{1}{\sqrt{\sin^2\theta + b^2\cos^2\theta}}\partial_\theta \right\}.
\]
Correspondingly, the Riemannian volume element becomes
\[
d\sigma = \cos\theta\sqrt{\sin^2\theta + b^2\cos^2\theta} d\varphi d\theta.
\]

In this framework, for function $\psi: \mathbb{S}^2 \mapsto \mathbb{R}$ and the velocity field $U = u(\varphi,\theta) \mathbf{e_\varphi} + v(\varphi,\theta) \mathbf{e_\theta}$, the basic operators have expressions as below:
\[
\text{grad }\psi = \frac{\partial_\varphi\psi}{\cos\theta}\mathbf{e_\varphi} + \frac{\partial_\theta\psi}{\sqrt{\sin^2\theta + b^2\cos^2\theta}}\mathbf{e_\theta} 
\]
\[
\nabla\cdot U = \frac{1}{\cos\theta}\partial_\varphi u + \frac{1}{\cos\theta\sqrt{\sin^2\theta + b^2\cos^2\theta}}\partial_\theta(\cos\theta v)
\]
\begin{equation}\label{Delta_formula}
    \Delta\psi = \frac{1}{\cos^2\theta}\partial_{\varphi\varphi}\psi - \frac{\tan\theta}{(\sin^2\theta + b^2\cos^2\theta)^2}\partial_\theta\psi + \frac{1}{\sin^2\theta + b^2\cos^2\theta}\partial_{\theta\theta}\psi.
\end{equation}
For a path $c(t)$ on the surface $\mathbb{S}^2$, \[\frac{d}{dt}\psi\circ c(t) = \operatorname{grad} \psi\cdot c'(t)\] gives a definition for the gradient. The formula for divergence comes from duality (see Richtmyer and Burdorf~\cite{richtmyer1978principles}), while
the Laplace–Beltrami operator $\Delta$ can be computed through the Voss-Weyl formula 
\[
\Delta=\frac{1}{\sqrt{|\operatorname{det}(g)|}} \sum_{i, j=1}^2 \frac{\partial}{\partial x^i}\left(g^{i j} \sqrt{|\operatorname{det}(g)|} \frac{\partial}{\partial x^j}\right),
\] 
where local coordinates $\left(x^1, x^2\right)=\left(\varphi, \theta\right)$ and $\left(g^{i j}\right)$ is the inverse of the Riemannian metric $g=\left(g_{i j}\right)$ in this coordinate system (see Grinfeld~\cite{grinfeld2013introduction}). 

Define the stream function $\psi(\varphi,\theta)$ such that 
\begin{equation}\label{def_psi}
    U = (u, v)^T = J \operatorname{grad}\psi = \left(-\frac{1}{\sqrt{\sin^2\theta + b^2\cos^2\theta}}\partial_\theta\psi, \quad\frac{1}{\cos\theta}\partial_\varphi\psi\right)^T,
\end{equation}
where $J$ is the counter-clockwise $90^\circ$ rotation matrix. The vorticity is $\Omega = \Delta\psi$, and the material derivative has the expression
\[
D_t = \partial_t + (U\cdot\nabla_{U}) = \partial_t + u\nabla_\mathbf{e_\varphi} + v\nabla_\mathbf{e_\theta}.
\]
When applying $D_t$ on the vorticity $\Delta\psi$, the formula becomes
\[
D_t\Delta\psi = \left( \partial_t + \frac{1}{\cos\theta\sqrt{\sin^2\theta+b^2\cos^2\theta}}[-\partial_\theta\psi\partial_\varphi + \partial_\varphi\psi\partial_\theta] \right) \Delta\psi.
\]
The Euler equations on the surface of an ellipsoid rotating with angular velocity $\omega$ can be written as
\begin{equation}\label{Ew_2}
     \left( \partial_t + \frac{1}{\cos\theta\sqrt{\sin^2\theta+b^2\cos^2\theta}}[-\partial_\theta\psi\partial_\varphi + \partial_\varphi\psi\partial_\theta] \right) \left(\Delta\psi+\frac{2\omega\sin\theta}{\sqrt{\sin^2\theta+b^2\cos^2\theta}}\right)
     =0 \tag{$\mathcal{E}_\omega$}.
\end{equation}
The equation~\eqref{Ew} coincides with the one in Taylor~\cite{taylor2016euler}, where more details about the Euler equations can be found.
As a consequence, the stationary Euler equations reduces to 
\[
[-\partial_\theta\psi\partial_\varphi + \partial_\varphi\psi\partial_\theta](\Delta\psi+\frac{2\omega\sin\theta}{\sqrt{\sin^2\theta+b^2\cos^2\theta}})=0.
\]

The coordinate chart in~\eqref{coor_chart} and the definition of the stream function $\psi$ in~\eqref{def_psi} bring in the artificial singularities at the North and South poles. For any $C^1$ function $\psi$, the continuity of velocity field $U$ at $N$ and $S$ implies 
\begin{equation}\label{boundary_condition_g}
    \lim_{\theta\to\pm\frac{\pi}{2}} \partial_{\varphi}\psi(\varphi,\theta) = 0.
\end{equation}
A periodic condition is also imposed on $\psi$, namely, 
\begin{equation}\label{boundary_condition_phi}
    \psi(\varphi,\theta) = \psi(\varphi+2\pi,\theta), 
\end{equation}
which ensures the existence of $\partial_\varphi\psi$ globally (see Constantin and Germain~\cite{pierre2022}). 

The Euler equations in terms of the velocity field $U$ becomes
\[
\left\{\begin{array}{l}D_t U+\frac{2 \omega \sin \theta}{\sqrt{\sin^2\theta+b^2\cos^2\theta}}JU=-\operatorname{grad} p \\ \operatorname{div} U=0\end{array}\right.
\]
where $p$ is the pressure field.

It is feasible to recover the velocity $U$ directly from the vorticity $\Omega$, though not locally. The method is stated in Dritschel and Boatto~\cite{dritschel2015motion}.

\section{Laplace-Beltrami operator}\label{Laplace}

In this part, we will introduce and present the spectral theorem of the Laplace-Beltrami operator $\Delta$ on a sphere and on an ellipsoid, which serves as the foundation of the Rossby-Haurwitz solutions.

\subsection{On a unit sphere} 

A spherical coordinate chart writes
\[
(\varphi, \theta) \in(-\pi, \pi) \times\left(-\frac{\pi}{2}, \frac{\pi}{2}\right) \mapsto(-\cos \varphi \cos \theta, \sin \varphi \cos \theta, \sin \theta).
\]
The Laplace-Beltrami operator $\Delta$ applied on a scalar function $\psi(\varphi,\theta)$ becomes 
\[
\Delta\psi = \frac{1}{\cos^2\theta}\partial_{\varphi\varphi}\psi - \tan\theta\partial_\theta\psi + \partial_{\theta\theta}\psi.
\]
The eigenvalues of $-\Delta$ on a unit sphere are $\{j(j+1), j \in \mathbb{N}\}$. Each corresponding eigenspace $\mathbb{E}_j$ is of dimension $(2j+1)$, of which the basis consists of the spherical harmonics
\[
X_j^m(\varphi, \theta)=(-1)^m \sqrt{\frac{(2 j+1)(j-m) !}{4 \pi(j+m) !}} P_j^m(\sin \theta) \mathrm{e}^{\mathrm{i} m \varphi}, \quad m=-j, \ldots, j,
\]
where $P_j^m$ is the associated Legendre polynomials given by
\[
P_j^m(x)=\frac{1}{2^j j !}\left(1-x^2\right)^{m / 2} \frac{\mathrm{d}^{j+m}}{\mathrm{~d}^{j+m} x}\left(x^2-1\right)^j, \quad m=-j, \ldots, j.
\]
Notice the symmetry 
\[
X_j^{-m} = (-1)^m \overline{X_j^m}
\]
and $X_j^{0}$ is zonal. 
Moreover, the spherical harmonics are orthonormal with respect to the inner product 
\[
\langle f_1, f_2 \rangle = \iint_{\mathbb{B}^2}f_1 \overline{f_2} d\sigma_B,
\]
where $\mathbb{B}^2$ stands for the unit sphere and the Riemannian volume element on $\mathbb{B}^2$ is
\[
d\sigma_B = \cos\theta d\varphi d\theta.
\]
More discussions about spherical harmonics can be found in Lea~\cite{lea2004mathematics}, Tung~\cite{tung1985group}, Constantin and Germain~\cite{pierre2022}.

\subsection{On a biaxial ellipsoid}\label{Laplace_elli}

The Laplace-Beltrami operator on the surface of an ellipsoid has been widely studied (see Pankratova~\cite{pankratova1970eigenfunctions}, Eswarathasan and Kolokolnikov~\cite{eswarathasan2022laplace}). 
In our setting, when applied on $\psi$, it becomes
\[
\Delta\psi = \frac{1}{\cos^2\theta}\partial_{\varphi\varphi}\psi - \frac{\tan\theta}{(\sin^2\theta + b^2\cos^2\theta)^2}\partial_\theta\psi + \frac{1}{\sin^2\theta + b^2\cos^2\theta}\partial_{\theta\theta}\psi.
\]
Here, we present the spectral theory for $\Delta$ on a biaxial ellipsoid close to a unit sphere provided by Eswarathasan and Kolokolnikov~\cite{eswarathasan2022laplace}.

\begin{lemma}\label{eig_Laplace}
     Let $L \in \mathbb{N}$ and $\beta \in \mathbb{R}\backslash\{0\}$. Consider the biaxial ellipsoid (major axis = 1; minor axis = b) where $b=1+\varepsilon \beta$ for $\varepsilon \in \mathbb{R}^{+}$ and $g_{\varepsilon}$ the metric from $\mathbb{R}^3$ restricted to the ellipsoid. 
     
     Then there exists $\varepsilon_0$ such that for all $\varepsilon<\varepsilon_0$ and $\Lambda \in \operatorname{spec}\left(-\Delta_g\right) \cap[0, L(L+1)]$, we have
\[
\Lambda=l(l+1)+\varepsilon \Lambda_1+O\left(\varepsilon^2\right)
\]

for $l=0,1,2, \ldots L$ and $m=-l, \ldots, l$ with $\Lambda_1$ being given by the explicit formula
\[
\Lambda_1 = (-\beta) \frac{2 l(l+1)}{(2 l+3)(2 l-1)}\left(2 l^2-2 m^2+2 l-1\right).
\]
Moreover, each $\Lambda$ has multiplicity two except for those whose expansion has $m=0$, which in this case corresponds to multiplicity one.
\end{lemma}
Furthermore, let $\mathbb{E}_{l,m}$ be the $(l,m)$-th eigenspace of $-\Delta$ associated with eigenvalue $\lambda_{l,m}$. When $m \ne 0$, the basis of $\mathbb{E}_{l,m}$ is of the form
\[
\left\{y_1(\theta)e^{im\varphi},y_2(\theta)e^{-im\varphi}\right\},
\]
for some smooth functions $y_1$ and $y_2$ that depends on $(l,m)$. 
The elements of the basis are orthonormal with respect to the inner product 
\[
\langle f_1, f_2 \rangle = \iint_{\mathbb{S}^2}f_1 \overline{f_2} d\sigma.
\]
Though there is no exact formula for $y_1$ and $y_2$ like in the case of a sphere, an approximation up to $O(\varepsilon^2)$ can be conducted (see Eswarathasan and Kolokolnikov~\cite{eswarathasan2022laplace}).

\section{Non-zonal Rossby-Haurwitz solutions}\label{Sec_Rossby}

\subsection{Rossby-Haurwitz solutions on a rotating sphere}\label{Rossby_sphere}

Extensive research about the classical Rossby-Haurwitz solutions on a rotating sphere has been conducted in the literature (see Craig~\cite{craig1945solution}, Haurwitz~\cite{haurwitz1940motion}, Constantin and Germain~\cite{pierre2022}, Rossby~\cite{rossby1939relation}, Verkley~\cite{verkley1984construction}). 
Here, we briefly state the primary results.

The stream functions of the stationary Rossby-Haurwitz solutions of degree $l$ are
\begin{equation}\label{RW_on_sphere}
    \psi(\varphi,\theta) = \alpha\sin\theta + Y_l(\varphi, \theta),
\end{equation}
where $Y_l$ is in the $l$-th eigenspace of $\Delta$ on a sphere and
\[
\alpha = \frac{2\omega}{2 - l(l+1)}.
\]
Here, we focus on the case when $l \geq 2$. It can be easily verified that $\psi(\varphi,\theta)$ solves the stationary Euler equations~\eqref{Ew_stationary1} in the case $b=1$.
The existence of the stationary solution is trivial since Craig~\cite{craig1945solution} discovered the exact expression utilizing spherical harmonics.

The travelling-wave Rossby-Haurwitz solutions with speed $c$ are of the form
\begin{equation}\label{tra_sphere}
    \psi(\varphi-ct,\theta,t) = \alpha\sin\theta + Y_l(\varphi-ct, \theta),
\end{equation}
where $\alpha$ is given by 
\[
\alpha = \frac{2\omega - l(l+1)c}{2 - l(l+1)}.
\]
The travelling-wave solutions can be obtained from the stationary solutions.
The stability properties have been discussed in details by Constantin and Germain~\cite{pierre2022}.
In particular, the non-zonal Rossby-Haurwitz solutions are not Lyapunov stable.

\subsection{Stationary Rossby-Haurwitz solutions on a rotating ellipsoid}\label{propose_elli_Rossby}

Due to the inaccuracy of modeling a planet as a perfect sphere, it is natural to generalize Rossby-Haurwitz solutions from a sphere to an ellipsoid and expect the instability property to be inherited. 
By employing the spectral theory of Laplace-Beltrami operator on an ellipsoid (see Section~\ref{Laplace_elli}), we are able to discover non-zonal solutions of the stationary Euler equations
\begin{equation}\label{stationary2}
        [-\partial_\theta\psi\partial_\varphi + \partial_\varphi\psi\partial_\theta]\left(\Delta\psi+\frac{2\omega\sin\theta}{\sqrt{\sin^2\theta+b^2\cos^2\theta}}\right)=0.
\end{equation}

Unlike the case on a sphere, for any $\alpha \in \mathbb{R}$, $\psi=\alpha\sin\theta+Y_l^m(\varphi,\theta)$ cannot be a solution for equation~\eqref{stationary2}. 
Thus, a natural generalization is to find some function $g(\theta) \in C^{3, \alpha}\left((-\frac{\pi}{2}, \frac{\pi}{2})\right)$ such that
\begin{equation}\label{elli_RW}
    \psi=g(\theta)+Y_l^m(\varphi,\theta), \quad m \neq 0
\end{equation}
solves~\eqref{stationary2} with the boundary conditions~\eqref{boundary_condition_g} that is
\begin{equation}\label{bon_con_psi}
    \lim_{\theta\to\pm\frac{\pi}{2}} \partial_{\varphi}\psi(\varphi,\theta) = 0.
\end{equation}
Here, $Y_l^m$ belongs to $\mathbb{E}_{l,m}$ associated with eigenvalue $\lambda_{l,m}$.

We propose the stationary Rossby-Haurwitz solutions on a rotating ellipsoid to be $\psi$ in~\eqref{elli_RW} solving~\eqref{stationary2} with the boundary conditions~\eqref{bon_con_psi} satisfied.
The first natural question is about the existence. 
We will provide the proof in Section~\ref{Existence}.  

\subsection{Existence of the stationary Rossby-Haurwitz solutions on a rotating ellipsoid}\label{Existence}
In this section, we will prove there exists stationary non-zonal Rossby-Haurwitz solutions 
\begin{equation}\label{RHwave}
    \psi = g(\theta) + Y_l^m(\varphi,\theta)
\end{equation}
of the stationary Euler equations
\begin{equation}\label{Ew_stationary}
     [-\partial_\theta\psi\partial_\varphi + \partial_\varphi\psi\partial_\theta]\left(\Delta\psi+\frac{2\omega\sin\theta}{\sqrt{\sin^2\theta+b^2\cos^2\theta}}\right)=0.
\end{equation}
The boundary conditions for $\psi$ has been discussed in~\eqref{boundary_condition_g}, which are 
\begin{equation}\label{Boundary_condition_psi}
    \lim_{\theta\to\pm\frac{\pi}{2}} \partial_{\theta}\psi(\varphi,\theta) = 0.
\end{equation}
Furthermore, since equation~\eqref{Ew_stationary} takes derivatives for three times, the regularity condition
\begin{equation}\label{smoothness_condition_psi}
    \psi \in C^{3, \alpha}\left((-\pi,\pi)\times\left(-\frac{\pi}{2}, \frac{\pi}{2}\right)\right)
\end{equation}
should also be imposed.

Plugging~\ref{RHwave} into~\ref{Ew_stationary}, we end up with a third order ODE for $g(\theta)$
\begin{equation}\label{ODE_10}
    -\lambda_{l,m} g'(\theta) = \left(\Delta g(\theta)\right)' + \left(\frac{2\omega\sin\theta}{\sqrt{\sin^2\theta+b^2\cos^2\theta}}\right)'.
\end{equation}
By the formula of $\Delta$ in~\eqref{Delta_formula}, the ODE~\eqref{ODE_10} can be derived from the following ODE 
\begin{equation}
    \begin{split}\label{ODE1}
        -\lambda_{l,m} g(\theta) &s= \frac{1}{\cos\theta \sqrt{\sin^2\theta+b^2\cos ^2\theta}}
        \left(\frac{\cos\theta}{\sqrt{\sin^2\theta+b^2\cos^2\theta}} g'(\theta)\right)'\\ 
        &\quad + \frac{2\omega\sin\theta}{\sqrt{\sin^2\theta+b^2\cos^2\theta}}.
    \end{split}
\end{equation}
The boundary conditions for $g$ are
\begin{equation}\label{boud_cond}
    \lim_{\theta\to\pm\frac{\pi}{2}} g'(\theta) = 0,
\end{equation}
and the regularity condition for $g$ is 
\[
 g \in C^{3, \alpha}\left(\left(-\frac{\pi}{2}, \frac{\pi}{2}\right)\right).
\]

\begin{theorem}\label{thm_existence}
    There exists function $g \in C^{3, \alpha}\left(\left(-\frac{\pi}{2},\frac{\pi}{2}\right)\right)$ that solves the ODE~\eqref{ODE1} with the boundary conditions~\eqref{boud_cond}.
    As a result, the corresponding $\psi = g(\theta) + Y_l^m(\varphi,\theta)$ is a solution of the stationary Euler equations~\eqref{Ew_stationary}.
\end{theorem}

\begin{proof}  

    With denotation $\rho(\theta) =\sqrt{\sin^2\theta+b^2\cos^2\theta}$, the ODE~\eqref{ODE1} can be expanded as
    \begin{equation}\label{Expansion}
        -\lambda_{l,m} g(\theta) = \left( -\frac{\tan\theta}{\rho^2(\theta)}+\frac{(1-b^2)\sin\theta\cos\theta}{\rho^4(\theta)} \right) g'(\theta) + \frac{1}{\rho^2(\theta)}g''(\theta) + \frac{2\omega\sin\theta}{\rho(\theta)}.
    \end{equation}
    Notice that in~\eqref{Expansion}, $-\lambda_{l,m}$ and $1/\rho^2(\theta)$ are even functions, while
    \[
    -\frac{\tan\theta}{\rho^2(\theta)}+\frac{(1-b^2)\sin\theta\cos\theta}{\rho^4(\theta)} \quad \text{and} \quad \frac{2\omega\sin\theta}{\rho(\theta)}
    \]
    are odd functions. 
    Thus, if we can find a solution $g_{left}(\theta)$ on $\left(-\frac{\pi}{2}, 0\right]$ with $g_{left}'\left(-\frac{\pi}{2}\right) = 0$ and $g_{left}(0) = 0$, the odd extension
    \[
    g(\theta) = \begin{cases}
                g_{left}(\theta) & \text{on $\theta \in \left(-\frac{\pi}{2}, 0\right]$}\\
                - g_{left}(-\theta) & \text{on $\theta \in \left(0, \frac{\pi}{2}\right)$,}
                \end{cases}
    \]
    solves the ODE~\eqref{ODE1}, with the boundary conditions~\eqref{boud_cond} satisfied. 
    
    The differentiability of the constructed $g(\theta)$ at $0$ is at least of three orders. 
    This is because the first and third derivatives of the odd function $g$ match automatically. 
    Moreover, $g''(0)=0$ can be implied from the expansion~\eqref{Expansion}, together with $g(0) = 0$.
    Therefore, the proof of Theorem~\ref{thm_existence} can be reduced to Theorem~\ref{reduce}.
\end{proof}

\begin{theorem}\label{reduce}
    There exists $C \in \mathbb{R}$, and a solution $g$ of the ODE~\eqref{ODE1} on $\left(-\frac{\pi}{2},0\right]$, such that 
    \[g \in C^{3,\alpha}\left(\left(-\frac{\pi}{2}, 0\right]\right), \quad g'\left(-\frac{\pi}{2}\right) = 0 \quad \text{and} \quad g(0) = 0.\]
\end{theorem}

\begin{proof}

    Writing the equation~\eqref{ODE1} into an integral equation, we have
    \begin{equation}\label{integral_equation}
        \begin{split}
            g(y)-C &= -\lambda_{l,m} \int_{-\frac{\pi}{2}}^{y} [ F(y)-F(\theta) ] \cos \theta\sqrt{\sin^2\theta+b^2\cos^2\theta} g(\theta)d\theta \\
            &\quad + \omega \int_{-\frac{\pi}{2}}^{y}\cos\theta \sqrt{\sin^2\theta+b^2\cos^2\theta} d\theta 
        \end{split}
    \end{equation}
    for any constant $C \in \mathbb{R}$, where $F$ satisfies
    \begin{equation}
         F(0) = 0, \quad F'(x) = \frac{\sqrt{\sin^2x+b^2\cos^2x}}{\cos x} \quad \text{on} \quad x\in \left(-\frac{\pi}{2},0\right]. 
    \end{equation}
    Define the kernel $K$ to be
    \begin{equation}
        K(y, \theta) := 
            \begin{cases}
              -\lambda_{l,m}[F(y)-F(\theta)]\cos\theta\sqrt{\sin ^2\theta+b^2\cos^2\theta} & \text{on $-\frac{\pi}{2}<\theta \leq y \leq 0$}\\
              0 & \text{on $-\frac{\pi}{2} = \theta \leq y \leq 0$};
            \end{cases}
    \end{equation}
    the function $r$ to be
    \begin{equation}
        r(y) := \omega \int_{-\frac{\pi}{2}}^{y}\cos\theta \sqrt{\sin^2\theta+b^2\cos^2\theta} d\theta \quad \text{for} \quad y\in \left[-\frac{\pi}{2},0\right];
    \end{equation}
    and the domain $D$ to be
    \[ D:= \left\{ (y, \theta): -\frac{\pi}{2} \leq \theta \leq y \leq 0 \}\right\}. \]
    The equation~\eqref{integral_equation} becomes
    \begin{equation}\label{int2}
        g(y) = \int_{-\frac{\pi}{2}}^{y}K(y,\theta)g(\theta)d\theta + r(y) + C.
    \end{equation}
    For any constant $C$:
    \begin{itemize}
        \item The integral equation~\eqref{int2} has a unique and continuous solution since $K(y,\theta)$ is continuous on $D$, by Lemma~\ref{lemma_voc}.
        \item Check the boundary condition at $-\frac{\pi}{2}$:
        \begin{multline} 
        g'(y) = -\lambda_{l,m} \int_{-\frac{\pi}{2}}^{y} \frac{\sqrt{\sin^2 y+b^2\cos^2 y}}{\cos y}\cos\theta\sqrt{\sin^2\theta+b^2\cos^2\theta} g(\theta)d\theta \\
        + {\omega}{\cos y} \sqrt{\sin ^2y+b^2\cos^2y}.
        \end{multline}
        As $y \to -\frac{\pi}{2}$, $g'(y) \to 0$ because $g$ is continuous and bounded on the closed domain $D$.
    \end{itemize}
    
    Now, we are going to prove $\exists$ $C \in \mathbb{R}$, such that the solution $g$ of~\eqref{int2} satisfies $g(0) = 0$.
    By Lemma~\ref{lemma_voc}, it can be implied that
    \begin{equation}
        g(0) = S\left(0, -\frac{\pi}{2}\right)C + \int_{-\frac{\pi}{2}}^{0} S(0,s)r'(s) ds,
    \end{equation}
    where $S\left(0,s\right)$ is the (unique) continuous solution of
    \[
    S(0,s) = 1 + \int_s^0 K(0,v)S(v,s)dv, \quad (0,s) \in D,
    \]
    and 
    \[
    r'(s) = \omega\cos s \sqrt{\sin^2s+b^2\cos^2s}.
    \]
    Since $S$ and $r'$ are determined, choosing 
    \begin{equation}
        C = -\frac{\int_{-\frac{\pi}{2}}^{0} S(0,s)r'(s)ds}{S\left(0, -\frac{\pi}{2}\right)}
    \end{equation}
    makes $g(0) = 0$. 
    The last thing to prove is $S\left(0, -\frac{\pi}{2}\right) \neq 0$.\\
 
    \noindent Proof by contradiction: Suppose $S\left(0, -\frac{\pi}{2}\right) = 0$, let $u(y) = S\left(y, -\frac{\pi}{2}\right)$, then by Lemma~\ref{lemma_voc},  $u$ satisfies
    \begin{equation}
        u(y) = 1 + \int_{-\frac{\pi}{2}}^{y} K(y,v)u(v)dv, \quad u(0) = 0.
    \end{equation}
    However, together with the following Volterra integral equation
    \begin{equation}\label{VIE_lemma2}
        \Tilde{u}(y) = 0 + \int_{-\frac{\pi}{2}}^{y} K(y,v)\Tilde{u}(v)dv
    \end{equation}
    and the solution $\Tilde{u} \equiv 0$, the assumption $u(0) = 0$ contradicts with Lemma~\ref{Lemma2}, by choosing $c$ in Lemma~\ref{Lemma2} to be 0.
\end{proof}

\subsection{Uniqueness of the stationary Rossby-Haurwitz solutions on a rotating ellipsoid}\label{Uniqueness}
We have proved the existence of the stationary Rossby-Haurwitz solutions $\psi=g(\theta)+Y_l^m(\varphi,\theta)$, for non-zonal $Y_l^m \in \mathbb{E}_{l,m}$ ($m\ne0$). 
In this section, we will prove the solution $\psi$, or equivalently $g$, is uniquely determined by $Y_l^m$ up to a constant. Consequently, the corresponding velocity $U$ is unique for a given $Y_l^m$.

\begin{theorem}\label{thm_uniquess}
    Given a non-zonal function $Y_l^m(\varphi,\theta) \in \mathbb{E}_{l,m}$, let $\psi(\varphi,\theta) = g(\theta) + Y_l^m(\varphi,\theta)$ for some function $g(\theta) \in C^{3,\alpha}\left((-\frac{\pi}{2}, \frac{\pi}{2})\right)$.
    If $\psi(\varphi,\theta)$ solves the stationary Euler equations
    \begin{equation}\label{Ew_stationary3}
        [-\partial_\theta\psi\partial_\varphi + \partial_\varphi\psi\partial_\theta]\left(\Delta\psi+\frac{2\omega\sin\theta}{\sqrt{\sin^2\theta+b^2\cos^2\theta}}\right)=0,
    \end{equation}
    with the boundary conditions
    \[
    \lim_{\theta\to\pm\frac{\pi}{2}} g'(\theta) = 0,
    \]
    $g(\theta)$ is unique up to a constant.
\end{theorem}

\begin{proof}

    Suppose $\psi_1 = g_1(\theta) + Y_l^m(\varphi,\theta)$ and $\psi_2 = g_2(\theta) + Y_l^m(\varphi,\theta)$ solves the equation~\eqref{Ew_stationary3}. Then, $\Tilde{g}(\theta) = g_1(\theta) - g_2(\theta)$ solves 
    \[
    -\lambda_{l,m} \Tilde{g}'(\theta) = (\Delta \Tilde{g}(\theta))',
    \]
    which is equivalent to 
    \begin{equation}\label{g_Tilde}
        -\lambda_{l,m} \Tilde{g}(\theta) = (\Delta \Tilde{g}(\theta))+C, \quad \forall C \in \mathbb{R}.
    \end{equation}
    The equation~\eqref{g_Tilde} has a trivial solution $\Tilde{g}(\theta) = -C/\lambda_{l,m}$. 
    
    We now prove this solution is unique.
    If there exists another $\Tilde{g}^*(\theta)$ solves the equation~\eqref{g_Tilde}, the difference $q(\theta) = \Tilde{g}(\theta) - \Tilde{g}^*(\theta)$ must satisfy
    \begin{equation}\label{q}
        -\Delta q(\theta) = \lambda_{l,m}q(\theta),
    \end{equation}
    which implies $q(\theta) \in \mathbb{E}_{l,m}$.
    However, since $\mathbb{E}_{l,m}$ has basis 
    \[
    \left\{y_1(\theta)e^{im\varphi},y_2(\theta)e^{-im\varphi}\right\}, \quad m\ne 0
    \]
    (see Section~\ref{Laplace_elli}), $q(\theta) \notin \mathbb{E}_{l,m}$ except for $q(\theta)\equiv 0$ because non-trivial functions in $\mathbb{E}_{l,m}$ must depend on $\varphi$. 
    This proves the solution $\Tilde{g}(\theta)$ of~\eqref{g_Tilde} is unique.
    Furthermore, it is implied that $g(\theta)$ is unique up to a constant.
\end{proof}

\subsection{Travelling-wave Rossby-Haurwitz solutions on a rotating ellipsoid}\label{sec_tra}

Similar to classical travelling-wave Rossby-Haurwitz solutions on a sphere, the travelling-wave solutions on an ellipsoid can also be obtained from the stationary solutions.
The construction is stated in Theorem~\ref{thm_travelling}.

\begin{theorem}\label{thm_travelling}
    Let $c \in \mathbb{R}$ and $\psi = g(\theta) + Y_l^m(\varphi,\theta)$ be a solution of the stationary Euler equations~\eqref{Ew_stationary1}. A traveling-wave solution $\psi_c$ with travelling speed $c$ is constructed as
    \begin{equation}\label{travelling_c}
        \psi_c(\varphi, \theta ,t) = g(\theta) + c \lambda_{l,m} f(\theta) + Y_l^m(\varphi-ct, \theta),
    \end{equation}
    where $f(\theta)$ is the solution of the following ODE
    \begin{equation}\label{Integral_f}
        -\lambda_{l,m} f(\theta) = \frac{1}{\cos\theta \sqrt{\sin^2\theta+b^2\cos ^2\theta}} \left(\frac{\cos\theta}{\sqrt{\sin^2\theta+b^2\cos^2\theta}} f'(\theta)\right)' + P(\theta),
    \end{equation}
    in which 
    \[
        P(\theta) = \int_{-\frac{\pi}{2}}^\theta\cos(s)\sqrt{\sin^2(s)+b^2\cos^2(s)}ds - \int_{-\frac{\pi}{2}}^0\cos(s)\sqrt{\sin^2(s)+b^2\cos^2(s)}ds.
    \]
\end{theorem}

\begin{proof}
    We aim to find some $\beta \in \mathbb{R}$ and $f(\theta) \in C^{3,\alpha}\left( (-\frac{\pi}{2}, \frac{\pi}{2})\right)$, such that
    \begin{equation}\label{psi_c}
        \psi_c(\varphi, \theta ,t) = g(\theta) + \beta f(\theta) + Y(\varphi-ct, \theta) 
    \end{equation}
    solves the Euler equations~\eqref{Ew} 
    \begin{equation}\label{Ew_1}
         \left( \partial_t + \frac{1}{\cos\theta\sqrt{\sin^2\theta+b^2\cos^2\theta}}[-\partial_\theta\psi\partial_\varphi + \partial_\varphi\psi\partial_\theta] \right) \left(\Delta\psi+\frac{2\omega\sin\theta}{\sqrt{\sin^2\theta+b^2\cos^2\theta}}\right) =0. \nonumber \tag{$\mathcal{E}_\omega$}
    \end{equation}

    Since $\psi = g(\theta) + Y_l^m(\varphi,\theta)$ is a stationary solution of~\eqref{Ew_stationary1}, $g$ should satisfy
    \begin{equation}\label{g_travelling}
            -\lambda_{l,m}g = \Delta g + \frac{2\omega\sin\theta}{\sqrt{\sin^2\theta+b^2\cos^2\theta}},
    \end{equation}
    which can be implied from the ODE~\eqref{ODE1}.
    With the help of~\eqref{g_travelling}, we can compute
    \begin{equation}\label{Tra_Lap}
        \Delta\psi_c = -\lambda_{l,m}g+\beta\Delta f - \lambda_{l,m}Y_l^m - \frac{2\omega\sin\theta}{\sqrt{\sin^2\theta+b^2\cos^2\theta}}.
    \end{equation}
    Plugging~\eqref{Tra_Lap} into~\eqref{Ew_1}, we have
    \begin{equation}\label{f_travelling}
            \lambda_{l,m}\left(\frac{\partial Y_l^m}{\partial \varphi}\right)c + \frac{1}{\cos\theta\sqrt{\sin^2\theta+b^2\cos^2\theta}}\left(\frac{\partial Y_l^m}{\partial \varphi}\right)\beta \left( \lambda_{l,m} f' + (\Delta f)'\right) = 0.
    \end{equation}
    Then, by setting $\beta = c \lambda_{l,m}$ and $f$ to satisfy
    \begin{equation}\label{ODE_f}
        \lambda_{l,m}f' + \left( \Delta f \right)' = -\cos\theta\sqrt{\sin^2\theta+b^2\cos^2\theta},
    \end{equation}
    the equation~\eqref{f_travelling} holds, which means $\psi_c(\varphi,\theta,t)$ is a solution of~\eqref{Ew}.

    Note that the ODE~\eqref{ODE_f} can be implied from the ODE~\eqref{Integral_f}, whose existence can be proved through a similar method in Theorem~\ref{thm_existence}, providing $P(\theta)$ is odd, smooth and bounded.
    Similarly, the solution $f$ is odd and belongs to $C^{3,\alpha}\left((-\frac{\pi}{2},\frac{\pi}{2})\right)$.
\end{proof}

\begin{theorem}\label{thm_tra_uniq}
    for a given $Y_l^m \in \mathbb{E}_{l,m}$ and speed $c$, the constructed travelling-wave solution
    \[
    \psi_c(\varphi, \theta ,t) = g(\theta) + c \lambda_{l,m} f(\theta) + Y_l^m(\varphi-ct, \theta)
    \]
    is unique in terms of the velocity field $U = J \operatorname{grad}\psi$. 
\end{theorem}
\begin{proof}
    The proof is similar to Theorem~\ref{thm_uniquess}. It is straightforward to verify the zonal part $g(\theta)+c\lambda_{l,m}f(\theta)$ of $\psi_c(\varphi,\theta,t)$ is unique up to a constant. Then, the velocity field is uniquely determined. 
\end{proof}

\subsection{Instability of non-zonal Rossby-Haurwitz solutions on a rotating ellipsoid}\label{sec_instability}

The stability properties of both stationary and travelling-wave solutions are of great interest.
Non-zonal Rossby-Haurwitz solutions have been shown to be Lyapunov unstable on a rotating sphere (see Constantin and Germain~\cite{pierre2022}).
We will establish an analogous result on a rotating ellipsoid. 
We will only prove the instability of traveling-wave solutions, as stationary solutions can be viewed as a special case of traveling-wave solutions with traveling speed zero.
In the following, we will use $\psi(t)$ to denote $\psi_c(\varphi,\theta,t)$ for notation simplicity.

\begin{theorem}\label{Ins_Ross}
    The Non-zonal Rossby-Haurwitz solutions travelling with speed $c$
    \[
    \psi_c = g(\theta) + c\lambda_{l,m}f(\theta) + Y_l^m(\varphi-ct,\theta)
    \]
    are Lyapunov unstable. 
    
    Specifically, for a given $\psi_c(t)$ with initial data $\psi_c(0)$, there exists a sequence of perturbed waves $\psi^n_c(t)$ with initial data $\psi^n_c(0) \to \psi_c(0)$, such that
    \[
    \underset{n\to\infty}{\liminf} \left\{ \sup_{t>0}||{\psi^n_c}(t)-\psi_c(t)||^2_{L^2(\mathbb{S}^2,d\sigma)} \right\} \geq \epsilon > 0,
    \]
    for some $\epsilon \in \mathbb{R}^+$.
\end{theorem}

\begin{proof}
    For a given travelling-wave solution
    \begin{equation}
        \psi_c(t) = g(\theta) + c\lambda_{l,m}f(\theta) + Y_l^m(\varphi-ct,\theta),
    \end{equation}
    with initial data $\psi_c(0)$, we construct a sequence of solutions $\psi^n_c(t)$ with initial data
    \begin{equation}
        {\psi_c^n(0)} = \psi_c(0) + \frac{1}{n}f(\theta) = g(\theta) + \left(c\lambda_{l,m}+\frac{1}{n}\right)f(\theta) + Y_l^m(\varphi,\theta).
    \end{equation}
    The solutions $\psi^n_c(t)$ of the Euler equations~\eqref{Ew} are
    \begin{equation}
        \psi^n_c(t) = g(\theta) + \left(c\lambda_{l,m}+\frac{1}{n}\right)f(\theta) + Y_l^m(\varphi-{c_n}t,\theta), \quad {c_n} = c + \frac{1}{n\lambda_{l,m}}
    \end{equation}

    Since $Y_l^m \in \mathbb{E}_{l,m}$, it can be decomposed to be
    \[ Y_l^m(\varphi, \theta) = a_1y_1(\theta) e^{im\varphi} + a_2y_2(\theta) e^{-im\varphi},\]
    for some $a_1,a_2\in\mathbb{R}$. Here, $\left\{y_1(\theta) e^{im\varphi},y_2(\theta) e^{-im\varphi}\right\}$ is the orthonormal basis of $\mathbb{E}_{l,m}$ (see Section~\ref{Laplace_elli}). 
    Without loss of generality, assuming $a_1\ne0$, $y_1 \not\equiv 0$, we have
    
    \begin{equation}
        \begin{split}
            &\quad\sup_{t>0}\left\|{\psi^n_c}(t)-\psi_c(t)\right\|^2_{L^2(\mathbb{S}^2,d\sigma)}\\
            &= \sup_{t>0} \left\| \frac{1}{n}f(\theta) + Y_l^m(\varphi-c_nt, \theta) - Y_l^m(\varphi-ct, \theta)\right\|^2_{L^2(\mathbb{S}^2,d\sigma)}\\
            &\geq\sup_{t>0} \left\| a_1y_1(\theta) (e^{im(\varphi-{c_n}t)} - e^{im(\varphi-ct)}) + a_2y_2(\theta) (e^{-im(\varphi-{c_n}t)} - e^{-im(\varphi-ct)})\right\|^2_{L^2(\mathbb{S}^2,d\sigma)}\\     &\quad - \frac{1}{n}\left\|f\right\|^2_{L^2(\mathbb{S}^2,d\sigma)}\\
            &= \sup_{t>0} \Bigr\{\left\| a_1y_1(\theta) (e^{im(\varphi-{c_n}t)} - e^{im(\varphi-ct)})\right\|^2_{L^2(\mathbb{S}^2,d\sigma)}\\
            &\quad + \left\|a_2y_2(\theta) (e^{-im(\varphi-{c_n}t)} - e^{-im(\varphi-ct)})\right\|^2_{L^2(\mathbb{S}^2,d\sigma)} \Bigr\}
            - \frac{1}{n}\left\|f\right\|^2_{L^2(\mathbb{S}^2,d\sigma)}\\
            &\geq  \sup_{t>0}\left\| a_1y_1(\theta) (e^{im(\varphi-c_nt)} - e^{im(\varphi-ct)}) \right\|^2_{L^2(\mathbb{S}^2,d\sigma)} - \frac{1}{n}\left\|f\right\|^2_{L^2(\mathbb{S}^2,d\sigma)}\\
            &= \sup_{t>0}\biggl\{ \left(\int_{-\frac{\pi}{2}}^{\frac{\pi}{2}}a_1^2y_1^2(\theta)\cos\theta\sqrt{\sin^2\theta+b^2\cos^2\theta} d\theta\right) \\
            &\quad\times\left(\int_{0}^{2\pi} \left|e^{im(\varphi-c_nt)} - e^{im(\varphi-ct)}\right|^2 d\varphi\right) \biggr\} - \frac{1}{n}\left\|f\right\|^2_{L^2(\mathbb{S}^2,d\sigma)}\\
            &=  \sup_{t>0}\left\{2\pi\left|1 - e^{im(c_n-c)t}\right|^2 \right\}\left( \int_{-\frac{\pi}{2}}^{\frac{\pi}{2}}a_1^2y_1^2(\theta)\cos\theta\sqrt{\sin^2\theta+b^2\cos^2\theta} d\theta \right)\\
            &\quad - \frac{1}{n}\left\|f\right\|^2_{L^2(\mathbb{S}^2,d\sigma)}\\
            &= 8\pi \left(\int_{-\frac{\pi}{2}}^{\frac{\pi}{2}}a_1^2y_1^2(\theta)\cos\theta\sqrt{\sin^2\theta+b^2\cos^2\theta} d\theta\right) - \frac{1}{n}\left\|f\right\|^2_{L^2(\mathbb{S}^2,d\sigma)}.
        \end{split}
    \end{equation}
    Since $||f||^2_{L^2(\mathbb{S}^2,d\sigma)} < \infty$, there exists $N \in \mathbb{N}$, such that for all $n \geq N$,
    \[ 
    \sup_{t \geq 0}||\psi^n_c(t)-\psi_c(t)||^2_{L^2(\mathbb{S}^2,d\sigma)} \geq 4\pi\int_{-\frac{\pi}{2}}^{\frac{\pi}{2}}a_1^2y_1^2(\theta)\cos\theta\sqrt{\sin^2\theta+b^2\cos^2\theta} d\theta > 0.
    \]
\end{proof}

\bmhead*{Data Availability}
Data sharing not applicable to this article as no datasets were generated or analysed during the current study.

\bmhead*{Statements and Declarations}
\bmhead*{Conflict of interest} 
The author states that there is no conflict of interest.

\begin{appendices}
\section{Results about integral equations}

The following two lemmas are used to prove Theorem~\ref{reduce}. In our setting, they are stated as follows.

\begin{lemma}[Variation-of-constants formula, see {Brunner~\cite[ch.\ 1.2.1, p.\ 10]{brunner2017volterra}}]\label{lemma_voc}
        Let $D$ to be a closed domain of kernel $K$, in the form
        \[
        D := \left\{(y, \theta): -\frac{\pi}{2} \leq \theta \leq y \leq 0 \right\}.
        \]
        For each constant $C$, assume that $r(y) \in C^1\left(\left[-\frac{\pi}{2}, 0\right]\right), \text{ s.t. }r\left(-\frac{\pi}{2}\right) = 0,$ and $K \in C(D)$. Then the unique solution $g \in C([-\frac{\pi}{2}, 0])$ of the Volterra integral equation
        \begin{equation}
            g(y) = \int_{-\frac{\pi}{2}}^{y} K(y,\theta)g(\theta) d\theta + r(y) + C, \quad y \in \left[-\frac{\pi}{2}, 0\right]
        \end{equation}
        is given by the variation-of-constants formula
        \begin{equation}\label{VOC}
            g(y) = S(y, -\frac{\pi}{2})C + \int_{-\frac{\pi}{2}}^{y} S(y,s)r'(s) ds, \quad y \in \left[-\frac{\pi}{2}, 0\right],
        \end{equation}
        where S(y,s) is the unique continuous solution of
        \begin{equation}\label{VIE_S}
            S(y,s) = 1 + \int_s^y K(y,v)S(v,s)dv, \quad (y,s) \in D.
            \end{equation} 
    \end{lemma}

    \begin{lemma}[see Diethelm and Ford~\cite{diethelm2012volterra}]\label{Lemma2}
        Let the equation 
        \begin{equation}\label{lemma2_equ}
            u(y) = u_0 + \int_{-\frac{\pi}{2}}^{y} K(y,v)u(v)dv
        \end{equation}
        satisfy the following assumptions:
        \begin{enumerate}
            \item for every $-\frac{\pi}{2} \leq \tau_1 \leq \tau_2 \leq y$, the integral 
            \[ \int_{\tau_1}^{\tau_2} K(y,v)u(v)dv \]
            and
            \[ \int_{-\frac{\pi}{2}}^{y} K(y,v)u(v)dv \]
            are continuous functions of y.
            \item $K(y,\cdot)$ is absolutely integrable for all $y \in [-\frac{\pi}{2}, 0]$
            \item there exsit points $-\frac{\pi}{2} = Y_0 < Y_1 < Y_2 < \dots < Y_N = 0, \quad Y_i \in \mathbb{R}$, such that with $y \geq Y_i$,
            \[ \int_{Y_i}^{min(y, Y_{i+1})} |K(y,v)|dv \leq \gamma < \frac{1}{2} \]
            \item for every $y \geq -\frac{\pi}{2}$
            \[ \lim_{\delta \to 0^+} \int_{y}^{y+\delta} \left|K(y+\delta,v)\right|dv = 0 \]
        \end{enumerate}

        Then, the equation~\eqref{lemma2_equ} has a unique continuous solution. Furthermore, for every $c \in \mathbb{R}$, there exists precisely one value of $u_0 \in \mathbb{R}$ for which the solution $u$ of~\eqref{lemma2_equ} satisfies $u(0) = c$.
    \end{lemma}




\end{appendices}

\bibliography{sn-bibliography}

\end{document}